\documentclass{amsart}
\usepackage{amsmath,amssymb,amsthm,comment,todonotes}
\usepackage{enumitem}
\theoremstyle{definition}
\newtheorem{definition}{Definition}[section]
\theoremstyle{plain}
\newtheorem{proposition}[definition]{Proposition}
\newtheorem{lemma}[definition]{Lemma}
\newtheorem{theorem}[definition]{Theorem}
\newtheorem{corollary}[definition]{Corollary}
\newtheorem*{theorem*}{Theorem}
\theoremstyle{remark}
\newtheorem{remark}[definition]{Remark}
\newcommand{\tr}{\textrm{tr}}

\title{Quartic differentials and harmonic maps in conformal surface geometry}

\author{Francis Burstall}
\address{(F. Burstall) Department of Mathematical Sciences, University of Bath, Bath. BA2 7AY. United Kingdom}
\email{feb@bath.ac.uk} 

\author{Emilio Musso}
\address{(E. Musso) Dipartimento di Matematica, Politecnico di Torino,
Corso Duca degli Abruzzi 24, I-10129 Torino, Italy}
\email{emilio.musso@polito.it}

\author{Mason Pember}
\address{(M. Pember) Dipartimento di Matematica, Politecnico di Torino,
Corso Duca degli Abruzzi 24, I-10129 Torino, Italy}
\email{mason.pember@polito.it}

\begin{document}
\maketitle

\begin{abstract}
We consider codimension 2 sphere congruences in pseudo-conformal geometry that are harmonic with respect to the conformal structure of an orthogonal surface. We characterise the orthogonal surfaces of such congruences as either $S$-Willmore surfaces, quasi-umbilical surfaces, constant mean curvature surfaces in 3-dimensional space forms or surfaces of constant lightlike mean curvature in 3-dimensional lightcones. We then investigate Bryant's quartic differential in this context and show that generically this is divergence free if and only if the surface under consideration is either superconformal or orthogonal to a harmonic congruence of codimension 2 spheres. We may then apply the previous result to characterise surfaces with such a property. 
\end{abstract}

\section{Introduction}
Bryant's quartic differential is a well known conformal invariant of a surface in the conformal $3$-sphere $S^{3}$, see~\cite{B1984}. Bryant proved that if a surface in $S^{3}$ is Willmore then this quartic differential is holomorphic. Voss~\cite{V1985} characterised surfaces for which this quartic differential is holomorphic as either Willmore surfaces or surfaces with constant mean curvature in some spaceform. See~\cite{B2012,BP2009,JMN2016} for a recent overview. An analogous problem was studied in~\cite{MN2016} in the context of Laguerre geometry. In this paper we shall define an analogous quartic differential for spacelike and timelike surfaces in the $(p,q)$-sphere. We shall then see that generically this operator is divergence free if and only if the surface is orthogonal to a harmonic congruence of codimension 2 spheres. 

Harmonic maps play a central role in conformal geometry and in particular the study of Willmore surfaces.  It is well known that a surface is Willmore if and only if its central sphere congruence (alternatively known as the conformal Gauss map or mean curvature sphere congruence) is harmonic~\cite{B1929,B1984,E1988,R1987}. In~\cite[Theorem 4.8]{M2006} it is shown that a pair of surfaces in $S^{n}$, considered as a map into the space of point pairs of $S^{n}$, is harmonic if and only if the surfaces are an adjoint pair of Willmore surfaces (see also~\cite{H1998} when $n=3$). Motivated by our result regarding Bryant's quartic differential, we shall investigate codimension 2 sphere congruences that are harmonic with respect to the conformal structure of an orthogonal surface. We characterise the orthogonal surfaces of such sphere congruences as either Willmore surfaces, quasi-umbilic surfaces, constant mean curvature surfaces in 3-dimensional space forms or surfaces of constant lightlike mean curvature in 3-dimensional lightcones. In this way, we extend Voss' characterisation to arbitrary codimension and signature.  

\textit{Acknowledgements.} 
The authors were partially supported by PRIN 2017 ``Real and Complex
Manifolds: Topology, Geometry and Holomorphic Dynamics''
(protocollo 2017JZ2SW5-004). The second and third authors also gratefully acknowledge the support of GNSAGA of INdAM and the MIUR grant ``Dipartimenti di Eccellenza'' 2018 - 2022, CUP: E11G18000350001, DISMA, Politecnico di Torino.

\section{Conformal geometry}
\label{sec:confgeom}
Given a vector space $V$ and a manifold $\Sigma$, we shall denote by $\underline{V}$ the trivial bundle $\Sigma\times V$. If $W$ is a vector subbundle of $\underline{V}$, we denote by $W^{(1)}$ the subset of $\underline{V}$ consisting of the images of sections of $W$ and derivatives of sections of $W$ with respect to the trivial connection on $\underline{V}$ and call $W^{(1)}$ the derived bundle of $W$. In general $W^{(1)}$ will not be a subbundle of $\underline{V}$, however, in many instances, we may assume that it is. 

Let $p,q\in \mathbb{N}$ and let $\mathbb{R}^{p+1,q+1}$ denote a $(p+q+2)$-dimensional vector space equipped with a non-degenerate symmetric bilinear form $(.,.)$ of signature $(p+1,q+1)$. Let $\mathcal{L}$ denote the lightcone of $\mathbb{R}^{p+1,q+1}$. As in M\"{o}bius geometry, the projective lightcone $\mathbb{P}(\mathcal{L})$ is a model of the conformal geometry that we are considering and the pseudo-orthogonal group $\textrm{O}(p+1,q+1)$ represents the group of transformations of this geometry. We shall refer to $\mathbb{P}(\mathcal{L})$ as the \textit{$(p,q)$-sphere}. 

It is well known that the exterior algebra $\wedge^{2}\mathbb{R}^{p+1,q+1}$ is isomorphic to the Lie algebra $\mathfrak{o}(p+1,q+1)$ of $\textrm{O}(p+1,q+1)$, i.e., the space of skew-symmetric endomorphisms of $\mathbb{R}^{p+1,q+1}$, via the isomorphism
\[ a\wedge b\mapsto (a\wedge b),\]
where for any $c\in\mathbb{R}^{p+1,q+1}$, 
\[ (a\wedge b)c = (a,c)b - (b,c)a.\]
We shall make silent use of this identification throughout this paper. 

As in the case of the conformal geometry of $S^{n}$ (see~\cite{BS2012, H2003}), one can break symmetry to study space form geometry. Choose a non-zero vector $\mathfrak{q}\in \mathbb{R}^{p+1,q+1}$. Then we define
\begin{equation} 
\label{eqn:spaceform}
\mathfrak{Q}^{p,q} := \{ y\in \mathcal{L}: (y,\mathfrak{q})=-1\}.
\end{equation}
This is isometric to a union of $(p+q)$-dimensional space forms with signature $(p,q)$ and constant sectional curvature $\kappa = - (\mathfrak{q},\mathfrak{q})$. $\mathfrak{q}$ is called the \textit{space form vector} of $\mathfrak{Q}^{p,q}$. In particular, if $(\mathfrak{q},\mathfrak{q})=0$, then $\mathfrak{Q}^{p,q}$ is isometric to $\mathbb{R}^{p,q}$: after fixing a complementary lightlike vector $\mathfrak{o}\in \mathbb{R}^{p+1,q+1}$ with $(\mathfrak{o},\mathfrak{q})=-1$, we have that $\langle \mathfrak{o},\mathfrak{q}\rangle^{\perp}\cong \mathbb{R}^{p,q}$, where angled brackets $\langle \cdot \rangle$ denotes the span of vectors. Now the orthogonal projection 
\begin{equation}
\label{eqn:oqproj} 
\pi:\mathbb{R}^{p+1,q+1}\to \langle \mathfrak{o},\mathfrak{q}\rangle^{\perp}, \quad y\mapsto y + (y,\mathfrak{q})\mathfrak{o} + (y,\mathfrak{o})\mathfrak{q}
\end{equation}
restricts to an isometry from $\mathfrak{Q}^{p,q}$ to $\langle \mathfrak{o},\mathfrak{q}\rangle^{\perp}$. Moreover, the image of $\mathfrak{Q}^{p,q}\cap \langle \mathfrak{o}\rangle^{\perp}$ under $\pi$ is the lightcone of $\langle \mathfrak{o},\mathfrak{q}\rangle^{\perp}$. 

\subsection{Sphere congruences}
$(r,s)$-spheres in $\mathbb{P}(\mathcal{L})$, that is totally umbilic submanifolds of signature $(r,s)$ that are maximal with respect to inclusion, are parametrised by linear subspaces of $\mathbb{R}^{p+1,q+1}$. Namely, given a subspace $V$ of $\mathbb{R}^{p+1,q+1}$ with signature $(r+1,s+1)$, we identify this with an $(r,s)$-sphere via the map $V\mapsto \mathbb{P}(V\cap\mathcal{L})$. 

\begin{definition}
An \textit{$m$-dimensional $(r,s)$-sphere congruence} is an $m$-dimensional manifold $\Sigma$ together with a signature $(r+1,s+1)$ subbundle $V$ of $\underline{\mathbb{R}}^{p+1,q+1}$ or, equivalently, a map from $\Sigma$ into $\textrm{Gr}_{r+1,s+1}(\mathbb{R}^{p+1,q+1})$. 
\end{definition}

Submanifolds of conformal geometry are maps $f:\Sigma\to \mathbb{P}(\mathcal{L})$ or, equivalently, null line subbundles $f\le \underline{\mathbb{R}}^{p+1,q+1}$. We shall make no distinction between these two objects. 

\begin{definition}
\label{def:envorth}
An \textit{enveloping submanifold} of a sphere congruence $V$ is a null line subbundle $f\le V$ satisfying $f^{(1)}\le V$. 

An \textit{orthogonal submanifold} of a sphere congruence $W$ is a null line subbundle $f\le W$ satisfying $f^{(1)} = f \oplus (f^{(1)}\cap W^{\perp})$. 
\end{definition}

\subsection{Spacelike and timelike immersions}

Suppose that $\Sigma$ is a 2-dimensional manifold and that $f:\Sigma\to \mathbb{P}(\mathcal{L})$ is a spacelike or timelike immersion, i.e., $f^{(1)}/f$ is a rank 2 bundle inheriting a non-degenerate metric from $\mathbb{R}^{p+1,q+1}$. Let $f^{(1)}/f$ have signature $(2-\epsilon, \epsilon)$ for $\epsilon \in\{0,1\}$. Since $f$ is null, it follows that $f^{(1)}\le f^{\perp}$. 

Suppose that $V$ is a $(2-\epsilon, \epsilon)$-sphere congruence enveloping $f$ and suppose that $U\le f^{(1)}$ is a rank 2 subbundle of $f^{(1)}$ that is complementary to $f$, i.e., $f^{(1)}=f\oplus U$. Then $U$ has signature $(2-\epsilon,\epsilon)$. Since $V\cap U^{\perp}$ has signature $(1,1)$, we may choose a null line subbundle $\hat{f}$, called a \textit{Weyl structure of $f$}, such that $V = (f\oplus \hat{f})\oplus_{\perp} U$. We now have a splitting of $\underline{\mathbb{R}}^{p+1,q+1}$: 
\begin{equation}
\label{eqn:weylsplit}
 \underline{\mathbb{R}}^{p+1,q+1} = f\oplus \hat{f}\oplus U \oplus V^{\perp}.
 \end{equation}
This yields a splitting of the trivial connection as 
\begin{equation}
\label{eqn:weylsplitcon}
d = D - \beta - \hat{\beta} + \mathrm{I\!I} + A
\end{equation}
where $D$ is a metric connection preserving $f$, $\hat{f}$, $U$ and $V^{\perp}$, and $\beta\in \Omega^{1}(\hat{f}\wedge U)$, $\hat{\beta}\in\Omega^{1}(f\wedge U)$, $\mathrm{I\!I}\in \Omega^{1}(U\wedge V^{\perp})$ and $A\in \Omega^{1}(f\wedge V^{\perp})$. 

\begin{remark}
\label{rem:Aobs}
Given a section $\hat{F}\in \Gamma \hat{f}$, we have that 
\[ d\hat{F} = D\hat{F} - \hat{\beta}\hat{F} + A\hat{F}.\]
Since $D\hat{F} - \hat{\beta}\hat{F} \in \Omega^{1}(V)$ and $A\hat{F}\in \Omega^{1}( V^{\perp})$ we see that $A$ is the obstruction to $\hat{f}$ being a second enveloping surface of $V$. 
\end{remark}

The sphere congruence $V$ also gives rise to a splitting of the trivial bundle 
\begin{equation}
\label{eqn:Vsplit} \underline{\mathbb{R}}^{p+1,q+1} = V\oplus V^{\perp},
\end{equation}
and thus a splitting of the trivial connection 
\begin{equation}
\label{eqn:Vsplitcon}
d = \mathcal{D}^{V} + \mathcal{N}^{V},
\end{equation}
where $\mathcal{D}^{V}$ is the sum of the induced connections on $V$ and $V^{\perp}$ and $\mathcal{N}^{V}\in \Omega^{1}(V\wedge V^{\perp})$. The condition that $f$ envelopes $V$ can the be characterised by $\mathcal{N}^{V}f \equiv 0$. In terms of the splitting~\eqref{eqn:weylsplit}, \eqref{eqn:weylsplitcon}, we have that 
\begin{equation} 
\label{eqn:VWeyl}
 \mathcal{D}^{V} = D - \beta - \hat{\beta} \quad \text{and}\quad \mathcal{N}^{V} = \textrm{I\!I} + A.
\end{equation}

\subsection{Central sphere congruence}
$\Sigma$ inherits a signature $(2-\epsilon, \epsilon)$ conformal structure from the metrics $(dF,dF)$ for  $F\in \Gamma f^{\times}$. We write $T\Sigma\otimes \mathbb{C} = T^{1,0}\Sigma\oplus T^{0,1}\Sigma$, where $T^{1,0}\Sigma$ and $T^{0,1}\Sigma$ are the complex null line subbundles of this conformal structure. When $\epsilon =0$ these are complex conjugate and when $\epsilon =1$ these are the complex span of real null line bundles. We may then split $U\otimes \mathbb{C} = U_{+}\oplus U_{-}$, where $U_{+}=\beta(T^{1,0}\Sigma)f$ and $U_{-}=\beta(T^{0,1}\Sigma)f$ are complex null line subbundles of $\underline{\mathbb{R}}^{p+1,q+1}\otimes \mathbb{C}$. The restriction of $\beta$ to $T^{1,0}\Sigma$ then satisfies $\beta^{1,0} \in \Omega^{1,0}(\hat{f}\wedge U_{+})$ and similarly $\beta^{0,1}\in \Omega^{0,1}(\hat{f}\wedge U_{-})$. Note that, since $U_{+}, U_{-}\le U\otimes \mathbb{C}$ are complex null lines subbundles, we have that they are parallel for the metric connection $D$. 

The conformal structure gives rise to a Hodge star operator $\star$ on $T^{*}\Sigma$ satisfying $\star^{2} = (-1)^{1-\epsilon}id$. By fixing an orientation, we may assume that $\star$ acts as $(-1)^{\frac{1-\epsilon}{2}}\textrm{id}$ on $(T^{1,0}\Sigma)^{*}$ and $-(-1)^{\frac{1-\epsilon}{2}}\textrm{id}$ on  $(T^{0,1}\Sigma)^{*}$. 

The \textit{central sphere congruence} of $f$ is given by
\[ V_{csc} :=  f^{(1)}\oplus \langle d_{Z_{+}}d_{Z_{-}}F\rangle,\]
where $F\in \Gamma f$, $Z_{+}\in \Gamma T^{1,0}\Sigma$ and $Z_{-}\in \Gamma T^{0,1}\Sigma$.  Note that $V_{csc}$ does not depend on choices. It is characterised by the following property:  

\begin{proposition}
\label{prop:csc}
Suppose that $V$ is a $(2-\epsilon, \epsilon)$ sphere congruence enveloped by $f$. Then
$V$ is the central sphere congruence of $f$ if and only if $\mathrm{I\!I}^{1,0}\in \Omega^{1,0}(U_{-}\wedge V^{\perp})$ and $\mathrm{I\!I}^{0,1}\in \Omega^{0,1}(U_{+}\wedge V^{\perp})$. 
\end{proposition}
\begin{proof}

Using the splitting \eqref{eqn:Vsplit} and \eqref{eqn:Vsplitcon}, we have that 
\[ d_{Z_{+}}d_{Z_{-}}F = \mathcal{D}^{V}_{Z_{+}}d_{Z_{-}}F + \mathcal{N}^{V}_{Z_{+}}d_{Z_{-}}F,\]
for $Z_{+}\in \Gamma T^{1,0}\Sigma$, $Z_{-}\in \Gamma T^{0,1}\Sigma$ and $F\in \Gamma f$. Since 
$\mathcal{D}^{V}_{Z_{+}}d_{Z_{-}}F\in \Gamma V$ it follows that $d_{Z_{+}}d_{Z_{-}}F\in \Gamma V$ if and only if $\mathcal{N}^{V}_{Z_{+}}d_{Z_{-}}F$. Since $\mathcal{N}^{V}f\equiv 0$, the latter is equivalent to $0 = \mathcal{N}^{V}_{Z_{+}}U_{-} = (\textrm{I\!I}_{Z_{+}} + A_{Z_{+}})U_{-}=\textrm{I\!I}_{Z_{+}}U_{-}$, using~\eqref{eqn:VWeyl}. Similarly, $d_{Z_{-}}d_{Z_{+}}F\in \Gamma V$ if and only if $\textrm{I\!I}_{Z_{-}}U_{+} = 0$. Finally, we have that $d_{Z_{-}}d_{Z_{+}}F \equiv d_{Z_{+}}d_{Z_{-}}F$ modulo $f^{(1)}$. Thus $d_{Z_{+}}d_{Z_{-}}F\in \Gamma V$ if and only if $\textrm{I\!I}_{Z_{+}}U_{-} = 0$ and $\textrm{I\!I}_{Z_{-}}U_{+} = 0$, proving the result. 
\end{proof}

From now on, let us assume that $V$ is the central sphere congruence of $f$. Since $U_{\pm}$ are null, Proposition~\ref{prop:csc} implies that 
\begin{equation}
\label{eqn:betaII}
[\beta^{1,0}\wedge \mathrm{I\!I}^{0,1}]= [\beta^{0,1}\wedge \mathrm{I\!I}^{1,0}]=0.
\end{equation}
Let us consider $\hat{\beta}$. We may write $\hat{\beta}= Q + \hat{\beta}_{0}$ where 
\begin{align*}
&Q^{1,0}\in \Omega^{1,0}(f\wedge U_{-}), \quad  Q^{0,1}\in \Omega^{0,1}(f\wedge U_{+}),\quad \text{and}\\ 
&\hat{\beta}^{1,0}_{0}\in \Omega^{1,0}(f\wedge U_{+}), \quad\hat{\beta}^{0,1}_{0}\in \Omega^{0,1}(f\wedge U_{-}).
\end{align*}
Notice that, since $U_{\pm}$ are rank 1 null subbundles, 
\begin{align} 
\label{eqn:betaQ} [\beta^{1,0}\wedge Q^{0,1}] = [\beta^{0,1}\wedge Q^{1,0}] &=0, \quad\text{and} \\
\label{eqn:IIhatbeta0} [\mathrm{I\!I}^{1,0}\wedge \hat{\beta}_{0}^{0,1}]=  [\mathrm{I\!I}^{0,1}\wedge \hat{\beta}_{0}^{1,0}] &=0 .
\end{align}
\begin{remark} 
We may identify $Q$ with the quadratic differential 
\[ q(X,Y) = \tr(f\to f: F\mapsto Q(X)\beta(Y)F) = - (\beta(Y), Q(X)).\]
Since $Q^{1,0} \in \Omega^{1,0}(f\wedge U_{-})$ and $Q^{0,1} \in \Omega^{0,1}(f\wedge U_{+})$, we have that $q = q^{2,0} + q^{0,2}$, where $q^{2,0} \in\Gamma(T^{*}\Sigma)^{2,0}$ and $q^{0,2}\in \Gamma(T^{*}\Sigma)^{0,2}$. 
\end{remark}

In addition to the equations~\eqref{eqn:betaII}, \eqref{eqn:betaQ} and \eqref{eqn:IIhatbeta0}, the flatness of the trivial connection
\[  d = D - \beta - Q - \hat{\beta}_{0} + \mathrm{I\!I} + A\]
gives rise to the Gauss-Codazzi-Ricci equations of the splitting:
\begin{align}
\label{eqn:curvD} R^{D} +[\beta\wedge \hat{\beta}_{0}] + \frac{1}{2}[\mathrm{I\!I}\wedge \mathrm{I\!I}] &= 0,\\
\label{eqn:Dbeta} d^{D}\beta &= 0, \\
\label{eqn:DQDhatbeta0AII} d^{D}Q + d^{D}\hat{\beta}_{0} - [A\wedge \mathrm{I\!I}] &= 0,\\
\label{eqn:DAQII} d^{D}A - [Q\wedge \mathrm{I\!I}] &=0, \quad \text{and} \\
\label{eqn:DII} d^{D}\mathrm{I\!I} - [\beta\wedge A] &=0,
\end{align}
where $R^{D}$ is the curvature tensor of the connection $D$. In the following sections we shall utilise the equations derived in this section when $U$ is a given rank 2 subbundle of $f^{(1)}$. 

\subsection{Superconformal surfaces}
In the conformal $n$-sphere $S^{n}$ there is a class of surfaces called superconformal surfaces that arise in the study of Willmore surfaces (see \cite{BFLPP2001, DT2009, DV2014, DV2015, MWW2017}). These are the surfaces $f$ for which $\textrm{I\!I}^{1,0}V$ is isotropic, where $V$ is the central sphere congruence of $f$. In this setting, since these surfaces are spacelike, one has that $\textrm{I\!I}^{0,1}V$ is also isotropic.

When a surface is timelike, it is possible for $\textrm{I\!I}^{1,0}V$ to be isotropic, independent of the nature of $\textrm{I\!I}^{0,1}V$. We therefore make the following definitions in our more general setting: 

\begin{definition}
\label{def:superconformal}
We call $p\in \Sigma$ a \textit{superconformal}  (respectively, \textit{half-superconformal}) point of $f$ if $(\textrm{I\!I}^{1,0}V)(p)$ and $(\textrm{I\!I}^{0,1}V)(p)$ are isotropic (exactly one of $(\textrm{I\!I}^{1,0}V)(p)$ or $(\textrm{I\!I}^{0,1}V)(p)$ is isotropic). 

$f$ is called \textit{superconformal} (respectively, \textit{half-superconformal}) if every point $p\in \Sigma$ is  superconformal (half-superconformal). 

$f$ is called \textit{nowhere superconformal} (respectively, \textit{nowhere half-superconformal}) if no point $p\in \Sigma$ is superconformal (half-superconformal). 
\end{definition}

\section{Orthogonal surfaces of codimension 2 harmonic sphere congruences}

Suppose that $f$ is an orthogonal surface to a codimension 2 sphere congruence $W$. Then, it follows from Definition~\ref{def:envorth}, that $f^{(1)} = f\oplus U$ where $U:= W^{\perp}$. We ask when $U^{\perp}$ is a harmonic map with respect to the conformal structure of $f$. The following theorem gives a complete answer to this question: 
\begin{theorem}
\label{thm:orthharmonic}
Orthogonal surfaces of codimension 2 harmonic sphere congruences are characterised, away from a nowhere dense subset of their domain, as either
\begin{itemize}
\item $S$-Willmore surfaces,
\item quasi-umbilical surfaces, 
\item surfaces of constant mean curvature in 3-dimensional space forms, or
\item surfaces of constant lightcone mean curvature in 3-dimensional lightcones. 
\end{itemize}
\end{theorem}

This section is devoted to proving Theorem~\ref{thm:orthharmonic}.

Let $V$ be the central sphere congruence of $f$ and let $\hat{f}$ be the Weyl structure of $f$ such that $V=(f\oplus \hat{f})\oplus_{\perp} U$.  Let us now consider the orthogonal splitting 
\[ \underline{\mathbb{R}}^{p+1,q+1} = U\oplus U^{\perp}\]
which yields the splitting of the trivial connection 
\begin{equation} 
\label{eqn:Usplit}
d = \mathcal{D}^{U} + \mathcal{N}^{U},
\end{equation}
where $\mathcal{D}^{U}$ is the sum of the induced connections on $U$ and $U^{\perp}$ and $\mathcal{N}^{U}\in \Omega^{1}(U\wedge U^{\perp})$. 
Comparing this with the splitting~\eqref{eqn:weylsplit} and~\eqref{eqn:weylsplitcon}, we have that 
\[ \mathcal{D}^{U} = D +A \quad \text{and} \quad \mathcal{N}^{U} = \mathrm{I\!I} - \beta -Q - \hat{\beta}_{0}.\]
The flatness of $d$ implies that $d^{\mathcal{D}^{U}}\mathcal{N}^{U}=0$. Thus $U^{\perp}$ is harmonic with respect to the conformal structure of $f$, i.e., $d^{\mathcal{D}^{U}}\star\mathcal{N}^{U}=0$, if and only if 
\begin{equation}
\label{eqn:DUNU}
d^{\mathcal{D}^{U}}(\mathcal{N}^{U})^{1,0}=0 = d^{\mathcal{D}^{U}}(\mathcal{N}^{U})^{0,1}. 
\end{equation}
The equation $d^{\mathcal{D}^{U}}(\mathcal{N}^{U})^{1,0}=0$ then splits as
\begin{align}
\label{eqn:DII10A}d^{D}\mathrm{I\!I}^{1,0} - [A\wedge \beta^{1,0}] &= 0,\\
\label{eqn:DQ10Dhatbeta010}- d^{D}Q^{1,0} - d^{D}\hat{\beta}_{0}^{1,0} + [A\wedge \textrm{I\!I}^{1,0}] &=0. 
\end{align}
Since $U_{\pm}, V^{\perp}$ are $D$-parallel subbundles, we have that $d^{D}\mathrm{I\!I}^{1,0}\in \Omega^{1,1}(U_{-}\wedge V^{\perp})$, whereas $[A\wedge \beta^{1,0}]\in \Omega^{1,1}(U_{+}\wedge V^{\perp})$. Hence, \eqref{eqn:DII10A} holds if and only if 
\begin{align}
\label{eqn:DII10} d^{D}\mathrm{I\!I}^{1,0} = 0,\\
\label{eqn:Abeta10} [A\wedge \beta^{1,0}] = 0. 
\end{align}
Since $\beta^{1,0}$ is nowhere zero, \eqref{eqn:Abeta10} holds if and only if $A^{0,1} = 0$. Equation~\eqref{eqn:DQ10Dhatbeta010} then reduces to 
\begin{align}
\label{eqn:DQ10}d^{D}Q^{1,0} &=0,\\
d^{D}\hat{\beta}_{0}^{1,0} &= 0,
\end{align}
using again the fact that $U_{\pm}$ are parallel subbundles of $D$. Similarly, one can show that $d^{\mathcal{D}^{U}}(\mathcal{N}^{U})^{0,1}=0$ if and only if 
\[A^{1,0} = d^{D}\mathrm{I\!I}^{0,1} = d^{D}Q^{0,1} = d^{D}\hat{\beta}_{0}^{0,1} =0.\]
Since $d^{D}Q^{1,0}\in \Omega^{2}(f\wedge U_{-})$ and $d^{D}Q^{0,1}\in \Omega^{2}(f\wedge U_{+})$, we have that $d^{D}Q^{1,0} = d^{D}Q^{0,1}= 0$ if and only if $d^{D}Q=0$. An analogous statement holds for $d^{D}\hat{\beta}_{0}$. Recalling Remark~\ref{rem:Aobs}, we have thus arrived at the following lemma: 

\begin{lemma}
\label{lem:harmcond}
$U^{\perp}$ is harmonic with respect to the conformal structure induced by $f$ if and only if $\hat{f}$ is a second enveloping surface of $V$,
\begin{align}
\label{eqn:DQ} d^{D}Q &=0, \quad \text{and} \\
\label{eqn:Dhatbeta0} d^{D}\hat{\beta}_{0} &=0 .
\end{align}
\end{lemma}

Suppose now that $U^{\perp}$ is harmonic with respect to the conformal structure induced by $f$. 
Together with~\eqref{eqn:betaQ} and \eqref{eqn:DAQII}, \eqref{eqn:DQ} implies that $dQ = 0$. Equivalently, $q$ is a divergence free quadratic differential. In particular, if $\epsilon = 0$, then $q$ is holomorphic. Since $Q\in \Omega^{1}(f\wedge f^{\perp})$, it follows that if $q\neq 0$ then $f$ is an isothermic surface (see~\cite{BDPP2011, BS2012}). 

Using the splitting~\eqref{eqn:Vsplit} and \eqref{eqn:Vsplitcon}, we have that 
\begin{align*}
d^{\mathcal{D}^{V}} (\mathcal{N}^{V})^{1,0} &= d^{D}\textrm{I\!I}^{1,0} - [\beta\wedge  \textrm{I\!I}^{1,0}] - [\hat{\beta}\wedge  \textrm{I\!I}^{1,0}] \\
&= d^{D}\textrm{I\!I}^{1,0} - [Q^{0,1}\wedge  \textrm{I\!I}^{1,0}] - [\hat{\beta}_{0}^{0,1}\wedge \textrm{I\!I}^{1,0}]\\
&= - [Q^{0,1}\wedge  \textrm{I\!I}^{1,0}] \\
&= - [Q\wedge (\mathcal{N}^{V})^{1,0}].
\end{align*}
Similarly, one can show that $\mathcal{D}^{V}( \mathcal{N}^{V})^{0,1} = - [Q\wedge (\mathcal{N}^{V})^{0,1}]$. Thus, 
\begin{equation} 
\label{eqn:Vconstrained}
\mathcal{D}^{V} \star\mathcal{N}^{V} = - [Q\wedge \star\mathcal{N}^{V}],
\end{equation}
implying that $f$ is a constrained Willmore surface with Lagrange multiplier $-q$ (see~\cite{BC2010i, BQ2014}). In particular, if $q$ vanishes then $f$ is a Willmore surface. 

\begin{lemma}
\label{lem:QIIvan}
Let $p\in \Sigma$ and suppose that $Q^{1,0}(p) = 0$ and $Q^{0,1}(p) \neq 0$ (or $Q^{1,0}(p) \neq 0$ and $Q^{0,1}(p) = 0$). Then $\textrm{I\!I}^{1,0}(p)=0$ (respectively, $\textrm{I\!I}^{0,1}(p)=0$). 
\end{lemma}
\begin{proof}
It follows from~\eqref{eqn:DAQII} that $[Q\wedge \textrm{I\!I}]=0$. Hence, 
\[ [Q^{1,0}(p)\wedge \textrm{I\!I}^{0,1}(p)] +  [Q^{0,1}(p)\wedge \textrm{I\!I}^{1,0}(p)] =0. \]
Since $Q^{1,0}\in \Omega^{1,0}(f\wedge U_{-})$ and $\textrm{I\!I}^{0,1}\in (U_{+}\wedge V^{\perp})$, one deduces that $[Q^{1,0}(p)\wedge \textrm{I\!I}^{0,1}(p)] = 0$ if and only if $Q^{1,0}(p)=0$ or  $\textrm{I\!I}^{0,1}(p)=0$. Similarly, $[Q^{0,1}(p)\wedge \textrm{I\!I}^{1,0}(p)]=0$ if and only if $Q^{0,1}(p)=0$ or  $\textrm{I\!I}^{1,0}(p)=0$. 
\end{proof}

It follows from Lemma~\ref{lem:QIIvan} that if $q$ is degenerate then $\mathcal{N}^{V}$ has rank one. Note that this can only happen in the case that $f$ is timelike, since in the spacelike case $q^{2,0} = \overline{q^{0,2}}$. In accordance with \cite{C2012iii,MPwip, T2012} we have the following definition: 

\begin{definition}
A timelike surface is called \textit{quasi-umbilical} if its central sphere congruence has rank 1. 
\end{definition}

\begin{remark}
In the case that $f$ is a spacelike surface, we have that $q$ is a holomorphic quadratic differential. Thus $q$ is either identically zero or has indefinite signature away from a discrete subset of $\Sigma$ where $q$ vanishes. In the case that $f$ is a timelike surface, a divergence free quadratic differential locally has the form $q = Udu^{2}+Vdv^{2}$ where $(u,v)$ are null coordinates and $U$ is a function of $u$ and $V$ is a function of $v$. Therefore, at a point $p\in \Sigma$, $q_{p}$ can be one of 3 types: $q_{p}$ can be zero, degenerate or non-degenerate. It is then not difficult to see that, off a nowhere dense subset of $\Sigma$, each point has a neighbourhood on which $q$ has constant type. 
\end{remark}

In summary, away from a nowhere dense subset of $\Sigma$, we may restrict to a subset of $\Sigma$ where $f$ falls into one of the three categories: 
\begin{enumerate}
\item if $q=0$ then $f$ is a Willmore surface,
\item if $q$ is degenerate then $f$ is a quasi-umbilical surface, and 
\item if $q$ is non-degenerate then $f$ is an isothermic constrained Willmore surface with respect to the same divergence free quadratic differential. 
\end{enumerate}
We shall examine these cases separately and see that stronger statements on $f$ hold. Firstly we see that the nature of $q$ can be characterised by the geometry of $U$. We have that 
\[ (\mathcal{N}^{U})^{1,0} = \textrm{I\!I}^{1,0} - \beta^{1,0} - Q^{1,0} - \hat{\beta}_{0}^{1,0}.\]
Hence, 
\[ (  (\mathcal{N}^{U})^{1,0},  (\mathcal{N}^{U})^{1,0}) = 2( \beta^{1,0},Q^{1,0}) = -2q^{2,0}.\]
Similarly, $ (  (\mathcal{N}^{U})^{0,1},  (\mathcal{N}^{U})^{0,1})=-2q^{0,2}$. Hence: 

\begin{lemma}
\label{lem:confUf}
The metric induced by $U$ is (weakly) conformally equivalent to the conformal structure induced by $f$ if and only if $q = 0$. 
\end{lemma}

\subsection{$U^{\perp}$ conformal harmonic}

Suppose now that $U^{\perp}$ is a harmonic congruence whose induced metric is conformally equivalent to the conformal structure induced by $f$, i.e., by Lemma~\ref{lem:confUf}, $Q=0$. Hence, $f$ is a Willmore surface. Moreover, since $A=0$, $\hat{f}$ is a second enveloping surface of $V$ by Lemma~\ref{lem:harmcond}. Using again that $Q=0$, yields that the induced conformal structure of $\hat{f}$ is weakly conformally equivalent to that of $f$ and induces the same orientation on $T\Sigma$. This implies that $f$ is $S$-Willmore in the sense of Ejiri~\cite{E1988}. 

Conversely, suppose that $f$ is $S$-Willmore with dual surface $\hat{f}$. Let $V$ be the central sphere congruence of $f$ and define $U:= V\cap (f\oplus \hat{f})^{\perp}$. Since $\hat{f}$ envelopes $V$ we have that $A = 0$. Then since $\hat{f}$ is weakly conformally equivalent to $f$ and induces the same orientation on $T\Sigma$, we have that $Q\equiv 0$. It then follows by~\eqref{eqn:DQDhatbeta0AII} that $d^{D}\hat{\beta}_{0} =0$. Therefore, by Lemma~\ref{lem:harmcond}, $U^{\perp}$ is harmonic and by Lemma~\ref{lem:confUf}, $U^{\perp}$ is conformally equivalent to $f$. 

In summary:

\begin{proposition}
$U^{\perp}$ is a conformal harmonic congruence of codimension 2-spheres with respect to the conformal structure of an orthogonal surface $f$ if and only if $f$ is S-Willmore with dual Willmore surface $\hat{f}$ as a second orthogonal surface. 
\end{proposition}

If, in addition to $q=0$, we have that $\hat{\beta}_{0}\equiv 0$, then $\hat{f}$ is constant. We may then choose a constant non-zero lightlike vector $\mathfrak{q}\in \Gamma \hat{f}$. This defines a pseudo-Euclidean space form $\mathfrak{Q}^{p,q}$ via~\eqref{eqn:spaceform} and $V$ corresponds to the tangent plane congruence of $f$. Since $V$ is also the central sphere congruence, it follows that $f$ is a minimal surface (see~\cite{BC2010i}). Conversely, if $f$ projects to a minimal surface in a pseudo-Euclidean space form $\mathfrak{Q}^{p,q}$ then the tangent plane congruence of this surface and the central sphere congruence coincide. Thus $\mathfrak{q}\in V$ and we define $\hat{f} = \langle \mathfrak{q}\rangle$. Since $\hat{f}$ is constant it follows that $\hat{\beta}_{0}=0$ and $q=0$. 

\begin{proposition}
\label{prop:minimal}
$\hat{f}$ envelopes $V$ with $q =\hat{\beta}_{0}=0$ if and only if $f$ projects to a minimal surface in a pseudo-Euclidean space form with space form vector $\mathfrak{q}\in \Gamma \hat{f}$. 
\end{proposition}

\subsection{Degenerate $q$}
Suppose that $q$ is non-zero and degenerate. Then we have already seen that this implies that $f$ is a a quasi-umbilical surface. 

Conversely, suppose that $f$ is a quasi-umbilical surface. Then $W:=\mathcal{N}^{V}(T\Sigma)V^{\perp}$ is a rank 1 null subbundle of $V$. Now $W^{\perp}\cap V$ has signature $(1,1,1)$ (i.e. a rank $3$ bundle spanned by an orthogonal basis consisting of a spacelike, timelike and degenerate vector field) and thus there exists a rank 1 null subbundle $\hat{f}\le W^{\perp}$ that only depends on 1 parameter and is complementary to $W$ and $f$. Then, since $\mathcal{N}^{V}\hat{f} = 0$, $\hat{f}$ envelopes $V$. Let $U:=f\oplus \hat{f}$. Then, since $\hat{f}$ only depends on 1-parameter, it follows that $d^{D}\hat{\beta}_{0}=d^{D}Q =0$, and thus $U^{\perp}$ is harmonic by Lemma~\ref{lem:harmcond}. 

\begin{proposition}
$U^{\perp}$ is a harmonic congruence of codimension 2-spheres with respect to the conformal structure of an orthogonal surface $f$ with $q$ degenerate if and only if $f$ is a quasi-umbilical surface. 
\end{proposition}

When $\hat{\beta}_{0} \equiv 0$, we have that $\hat{\beta}=Q$. This implies that $\hat{\beta}\hat{f} = \textrm{I\!I}V^{\perp}$. Such surfaces are called \textit{exceptional quasi-umbilical surfaces}, see~\cite{MPwip}. They satisfy the property that $V^{\perp}$ belongs to a constant $(p+q)$-dimensional subspace of $\underline{\mathbb{R}}^{p+1,q+1}$. Namely, if $q^{1,0}\equiv 0$, then this subspace is given by
\[ W:= U_{+}\oplus \hat{f}\oplus V^{\perp}.\]

\begin{proposition}
\label{prop:excquasi}
$\hat{f}$ envelopes $V$ with $q$ degenerate and $\hat{\beta}_{0}=0$ if and only if $f$ is an exceptional quasi-umbilical surface. 
\end{proposition}

\subsection{Non-degenerate $q$}
In~\cite{BPP2002} the surfaces in $S^{3}$ that are simultaneously isothermic and constrained Willmore with respect to the same quadratic differential $q$ are characterised as constant mean curvature surfaces in 3-dimensional Riemannian space forms (see also~\cite{BPP2008,R1997}). In this subsection we shall see that an analogous result is true in the $(p,q)$-sphere, yielding a characterisation of orthogonal surfaces of a harmonic congruence of $2$-spheres when $q$ is non-degenerate. 

Suppose that $U^{\perp}$ is a harmonic sphere congruence with $q$ non-degenerate. Since this implies that $f$ is isothermic, we have that $V^{\perp}$ is flat. 

\begin{lemma} 
\label{lem:IIbracket}
One may write
\[ \textrm{I\!I} = [Q, \xi],\]
for some section $\xi\in \Gamma (\hat{f}\wedge V^{\perp})$. 
\end{lemma}
\begin{proof}
Fix two nowhere zero one forms $\omega_{+}\in \Omega^{0,1}(U_{+})$ and $\omega_{-}\in \Omega^{1,0}(U_{-})$. Then we may write 
\[ \textrm{I\!I} = \omega_{-}\wedge N_{-} + \omega_{+}\wedge N_{+}\]
for some sections $N_{+}, N_{-}\in \Gamma (V^{\perp}\otimes \mathbb{C})$. We may also write 
\[ Q = \lambda \, F\wedge \omega_{-} + \mu\, F\wedge \omega_{+},\]
for some nowhere zero functions $\lambda$ and $\mu$, and a nowhere zero section $F\in \Gamma f$. Now the condition $[Q\wedge \textrm{I\!I}]=0$ from~\eqref{eqn:DAQII} implies that 
\[ \lambda\, N_{+} - \mu\, N_{-} = 0.\]
Then $N_{+} = \frac{\mu}{\lambda} N_{-}$. Hence, 
\[ \textrm{I\!I} = \frac{1}{\lambda} (\lambda\, \omega_{-}\wedge N_{-} + \mu\, \omega_{+}\wedge N_{-}).\]
Therefore, $\textrm{I\!I} = [Q,\xi]$ with $\xi := -\frac{1}{\lambda}\hat{F}\wedge N_{-}$, where $\hat{F}\in \Gamma \hat{f}$ such that $(F,\hat{F})=-1$. 
\end{proof}

Assume that $\Sigma$ is umbilic-free, i.e., $\textrm{I\!I}$ is nowhere zero. By Lemma~\ref{lem:IIbracket}, there exists $\xi \in \Gamma(\hat{f}\wedge V^{\perp})$ such that $\textrm{I\!I} = [Q, \xi]$. Since we are only considering points where $\textrm{I\!I}$ is non-zero, we may write $\xi = \hat{F}\wedge N$ for some nowhere zero sections $\hat{F}\in \Gamma \hat{f}$ and $N\in \Gamma V^{\perp}$. Let $L:= \langle N\rangle$. By  \eqref{eqn:DII10} and \eqref{eqn:DQ10}, we deduce that 
\[0 = [Q^{1,0}\wedge d^{D}\xi].\]
Since $Q^{1,0}$ is nowhere zero, this implies that $(d^{D}\xi)^{0,1} = 0$. Similarly, using that $d^{D}Q^{0,1} = d^{D}\textrm{I\!I}^{0,1}=0$, we can deduce that $(d^{D}\xi)^{1,0} = 0$. Hence 
\begin{equation}
\label{eqn:Dxi}
0 = d^{D}\xi = d^{D}\hat{F}\wedge N + \hat{F}\wedge d^{D}N.
\end{equation}
Thus $d^{D}N \in \Omega^{1}(L)$ and $d^{D}\hat{F}\in \Omega^{1}(\hat{f})$. Hence, $L$ is a parallel subbundle of $V^{\perp}$. Since $D|_{V^{\perp}}$ is flat, there exists a section $\tilde{N}\in \Gamma L$ such that $d^{D}\tilde{N}=0$. After rescaling $\hat{F}$, we may assume that $N = \tilde{N}$. It then follows from~\eqref{eqn:Dxi} that $d^{D}\hat{F}=0$. In particular, this implies that $R^{D}|_{f\oplus \hat{f}}=0$. Moreover, $R^{D}|_{V^{\perp}}=0$ and, since $U^{\perp} = f\oplus \hat{f}\oplus V^{\perp}$, it follows that $U^{\perp}$ is flat.

Now set $W:= V\oplus \langle N\rangle$. This is a 5-dimensional subbundle of $\underline{\mathbb{R}}^{p,q}$. Since $\textrm{I\!I}\in \Omega^{1}(U\wedge \langle N\rangle)$ and $N$ is $D$-parallel, it follows that $W$ is constant. 

\begin{lemma}
\label{lem:constW}
Away from umbilic points we have that $U^{\perp}$ is flat, i.e., $U^{\perp}$ is a cyclic congruence\footnote{For more information on cyclic congruences, see~\cite{H1996}.}, and
$f$ takes values in a 5-dimensional constant subspace $W$ of $\underline{\mathbb{R}}^{p,q}$. 
\end{lemma}

Since $f\oplus \hat{f}$ is flat, it follows from~\eqref{eqn:curvD} that $[\beta\wedge \hat{\beta}_{0}]|_{f\oplus \hat{f}} =0$. We can then deduce that 
\begin{equation}
\label{eqn:hatbeta0mu}
\hat{\beta}_{0}\hat{F} = \mu\, \beta F,
\end{equation}
for some function $\mu$ and $F\in\Gamma f$ such that $(\hat{F},F)=-1$. Note that since $\hat{F}$ is $D$-parallel, it follows that $F$ is $D$-parallel. By~\eqref{eqn:Dbeta} and~\eqref{eqn:Dhatbeta0} we have that  
\[ 0 = d^{D}(\hat{\beta}_{0}\hat{F}) = d^{D}(\beta F).\]
Hence, $\mu$ is constant. 

Since $N$ is $D$-parallel and $D$ is a metric connection, we have that $(N,N)$ is constant. Assume that $(N,N)\not\equiv 0$ or, equivalently, $(\textrm{I\!I},\textrm{I\!I})\not\equiv 0$. Without loss of generality, assume that $(N,N)=\pm 1$. Set 
\[ \mathfrak{q} := \hat{F} -\mu \, F - (N,N) N.\]
Then
\begin{align*}
d\mathfrak{q} &= d\hat{F} - \mu dF - (N,N) dN \\
&=   -Q\hat{F} + \hat{\beta}_{0}\hat{F} - \mu\,\beta F - (N,N)\textrm{I\!I}N\\
&=  -Q\hat{F} - (N,N)[Q,\hat{F}\wedge N]N\\
&= -Q\hat{F} + (N,N)^{2} Q\hat{F}\\
&= 0,
\end{align*}
since $(N,N)=\pm 1$. We then conclude from \cite{BPP2002} that $f$ has constant mean curvature $H = - (N,\mathfrak{q}) = 1$ in the space form defined by $\mathfrak{q}$ with sectional curvature $\kappa = - (\mathfrak{q},\mathfrak{q}) = - 2\mu - 1$. We thus arrive at the following lemma:

\begin{lemma}
\label{lem:harmCMC}
Assume that $(\textrm{I\!I},\textrm{I\!I})\not\equiv 0$. Then $f$ projects to a surface of constant mean curvature in a 3-dimensional space form. 
\end{lemma}

We now wish to see the converse of Lemma~\ref{lem:harmCMC}. Assume that $(p+1,q+1)=(4-\epsilon,1+\epsilon)$ or $(p+1,q+1)=(3-\epsilon,2+\epsilon)$ and suppose that $f$ projects to a constant mean curvature surface $F$ in a 3-dimensional space form defined by $\mathfrak{q}$ with $H\equiv 1$. That is, there exists $N\in \Gamma V^{\perp}$ with $(N,N)=\pm 1$, $(\mathfrak{q},N)=-1$. Fix $U$ such that 
\[ U^{\perp} = f\oplus V^{\perp} \oplus\langle \mathfrak{q}\rangle.\]
 We then obtain a line bundle $\hat{f}$ such that $V\cap U^{\perp} = f\oplus\hat{f}$. We may now write 
\[ \mathfrak{q} = \hat{F} - \mu F - (N,N)N\]
for some constant $\mu$ with $\hat{F}\in \Gamma \hat{f}$ such that $(F,\hat{F})=-1$. As before, $U$ induces a splitting of the trivial connection $d=\mathcal{D}^{U} + \mathcal{N}^{U}$ with 
\[ \mathcal{D}^{U} = D + A \quad \text{and}\quad \mathcal{N}^{U} = - \beta -\hat{\beta} + \textrm{I\!I}.\]
Since $\mathfrak{q}\in U^{\perp}$ is constant, we have that $\mathcal{D}^{U}\mathfrak{q} =0$ and $\mathcal{N}^{U}\mathfrak{q}=0$. Hence, 
\begin{equation} 
\label{eqn:DUq}
0 = \mathcal{D}^{U}\mathfrak{q} = D\hat{F} - \mu DF - (N,N)DN + A\hat{F} - (N,N) AN.
\end{equation}
Since, $(N,N)=\pm 1$ and $V^{\perp}$ is a 1-dimensional bundle, it follows that $DN=0$. Hence, $A\hat{F} = 0$, from which we deduce that $A\equiv 0$. Moreover, we deduce from~\eqref{eqn:DUq} that $D\hat{F}=0$. Since $(F,\hat{F})=-1$, we then have that $DF=0$. 

On the other hand, $\mathcal{N}_{U}\mathfrak{q} = 0$ implies that 
\begin{equation} 
\label{eqn:NUq} 
0 = \mathcal{N}^{U}\mathfrak{q}  = \mu\beta F -(\hat{\beta}_{0}+Q)\hat{F}- (N,N) \textrm{I\!I}N.
\end{equation}
By evaluating the $(1,0)$ and $(0,1)$ parts of~\eqref{eqn:NUq}, we have that 
\begin{align}
\label{eqn:mubetaF} 0 &=  \mu\beta F -\hat{\beta}_{0} \hat{F}\\
\label{eqn:QhatF} 0 &= - Q\hat{F} - (N,N) \textrm{I\!I}N.
\end{align}
By~\eqref{eqn:Dbeta}, we have that $d^{D}\beta = 0$. Then since $F$ and $\hat{F}$ are $D$-parallel, we learn from~\eqref{eqn:mubetaF} that $d^{D}\hat{\beta}_{0} =0$. We then have by~\eqref{eqn:DQDhatbeta0AII} that $d^{D}Q=0$. Thus, by applying Lemma~\ref{lem:harmcond}, we deduce that $U^{\perp}$ is harmonic. If we assume that $\textrm{I\!I}$ is nowhere zero, then~\eqref{eqn:QhatF} implies that $Q$ is non-degenerate. 

\begin{remark}
Geometrically, $U^{\perp}$ represents the congruence of geodesics that intersect $F$ orthogonally, i.e., it is the congruence of normal geodesics. 
\end{remark}

We have thus arrived at the following proposition: 

\begin{proposition}
\label{prop:qnondeg}
$U^{\perp}$ is a harmonic congruence of codimension 2-spheres with respect to the conformal structure of an orthogonal surface $f$ with $q$ non-degenerate and $(\textrm{I\!I},\textrm{I\!I})\not\equiv 0$ if and only if $f$ projects to a constant mean curvature surface in a 3-dimensional space form with $U^{\perp}$ being the congruence of normal geodesics. 
\end{proposition}

In the proof of Proposition~\ref{prop:qnondeg} we observe that $H \equiv 1$ and the sectional curvature $\kappa = -2\mu -1$. Thus, $H^{2} + \kappa = -2\mu$. When $\mu = 0$ we then have that $H^{2}+\kappa =0$ which corresponds to the case that $f$ projects to a CMC-$1$ surface in hyperbolic space $\mathbb{H}^{3}$, $\mathbb{H}^{2,1}$ or $\mathbb{H}^{1,2}$. Such surfaces can be constructed from holomorphic data, see~\cite{HMN2001}. On the other hand, when $\mu = 0$, we have that $\hat{\beta}_{0}=0$, i.e., $\hat{f}$ has the same conformal structure as $f$ and induces opposite orientation on $T\Sigma$. Thus, $\hat{f}$ is a Darboux transform of $f$ (see for example~\cite{BDPP2011}). 

\begin{proposition}
\label{prop:CMCholo}
Suppose that and $(\textrm{I\!I}, \textrm{I\!I})\not \equiv 0$. Then $\hat{f}$ envelopes $V$ with $q$ non-degenerate and $\hat{\beta}_{0}=0$ if and only if $f$ projects to a CMC-$1$ surface in a 3-dimensional space form with $\kappa = -1$. 
\end{proposition}

We now wish to study the case when $(\textrm{I\!I}, \textrm{I\!I})\equiv 0$. Note that, by Definition~\ref{def:superconformal}, $f$ is a superconformal surface. Suppose that $U^{\perp}$ is a harmonic sphere congruence with $q$ non-degenerate. Then by Lemma~\ref{lem:constW} we have that $f$ is contained in a constant $5$-dimensional space $W$. This space is degenerate and there exists a vector $\mathfrak{o}\in W$ such that $W\perp \mathfrak{o}$. 
Now $H:= U\oplus \mathcal{N}^{U}(T\Sigma)U$ is a constant 4-dimensional subspace of $W$ with signature $(3-\epsilon, 1+\epsilon)$. We may choose a complementary lightlike vector $\mathfrak{q}\in H^{\perp}$ with $(\mathfrak{o},\mathfrak{q})=-1$. Without loss of generality, we shall assume that $(p+1,q+1) = (4-\epsilon, 2+\epsilon)$, so that 
\[ \mathbb{R}^{p+1,q+1} = H\oplus_{\perp} \langle \mathfrak{o},\mathfrak{q}\rangle.\]
The projection~\eqref{eqn:oqproj} then yields an isometry between $H$ and $\mathfrak{Q}^{p,q}$ as in~\eqref{eqn:spaceform}. Since $dF \in \Omega^{1}(U)$, we may rescale $F$ so that $(F,\mathfrak{q})=-1$. Moreover, since $F\perp \mathfrak{o}$, $x:=\pi(F)$ takes values in the 3-dimensional lightcone $\mathcal{L}^{3}$ of $H$. Since $f\oplus \hat{f}$ is flat, we may choose sections $F\in \Gamma f$ and $\hat{F}\in \Gamma \hat{f}$ such that $(dF,\hat{F})=0$ and $(F,\hat{F})=-1$. Defining $x^{l} = 2\pi(\hat{F})$ we have that $x^{l}$ take values $\mathcal{L}^{3}$  and $\langle x,x^{l}\rangle$ is a flat bundle with $(x,x^{l})=-2$. Thus the ordered pair formed by $x$ and $x^{l}$ is a Legendrian submanifold $\Sigma\to \mathcal{L}^{3}\times \mathcal{L}^{3}$ in the sense of~\cite{I2009} and $x^{l}$ is a lightcone normal vector to $x$. Moreover, the condition~\eqref{eqn:hatbeta0mu} implies that $x$ is a surface of constant lightcone mean curvature $H_{l} = - \frac{\mu}{2}$. 

Conversely, suppose that $\mathfrak{o},\mathfrak{q}\in \mathbb{R}^{4-\epsilon, 2+\epsilon}$ and consider a surface of constant lightcone mean curvature $x$ in the 3-dimensional lightcone of $\langle \mathfrak{o},\mathfrak{q}\rangle^{\perp}$. Let $x^{l}$ denote the lightcone normal vector $x$. Then $x$ lifts to a surface $F:= x +\mathfrak{o}$ in $\mathfrak{Q}^{3-\epsilon,1+\epsilon}$ and one can check that $U:= dF(T\Sigma)$ gives rise to a codimension 2 harmonic sphere congruence $U^{\perp} = \langle x,x^{l},\mathfrak{o},\mathfrak{q}\rangle$ for which $f=\langle F\rangle$ is orthogonal. 

\begin{proposition}
$U^{\perp}$ is a harmonic congruence of codimension 2-spheres with respect to the conformal structure of an orthogonal surface $f$ with $q$ non-degenerate and $(\textrm{I\!I},\textrm{I\!I})\equiv 0$ if and only if $f$ projects to a surface of constant lightcone mean curvature in a 3-dimensional lightcone with $\hat{f}$ projecting to the lightcone normal vector of $f$. 
\end{proposition}

This concludes the proof of Theorem~\ref{thm:orthharmonic}. 

\section{Bryant's quartic differential}
In this section we shall derive a quartic differential $\mathcal{Q}$ for a surface in the $(p,q)$-sphere, analogous to Bryant's quartic differential in $S^{3}$. We then investigate the vanishing of $\mathcal{Q}$ and when $\mathcal{Q}$ is divergence free, giving us generalisations of the results of~\cite{B2012,BP2009,V1985} to higher codimension and arbitrary signature.  

Let $V$ be the central sphere congruence of $f$ and recall that we may split the trivial connection as 
\[ d = \mathcal{D}^{V}+\mathcal{N}^{V},\]
where $\mathcal{D}^{V}$ is the sum of the induced connections on $V$ and $V^{\perp}$ and $\mathcal{N}^{V}\in \Omega^{1}(V\wedge V^{\perp})$. Define a quartic differential $\mathcal{Q}\in \Gamma ((T^{4,0}\Sigma)^{*}\oplus  (T^{0,4}\Sigma)^{*})$ by 
\[ \mathcal{Q}(Z,Z,Z,Z) := ( \mathcal{D}^{V}_{Z}\mathcal{N}_{Z}^{V},\mathcal{D}^{V}_{Z}\mathcal{N}^{V}_{Z}),\]
for $Z\in \Gamma T^{1,0}\Sigma$ or $Z\in \Gamma T^{0,1}\Sigma$. 

\begin{definition}
$\mathcal{Q}\in \Gamma ((T^{4,0}\Sigma)^{*}\oplus  (T^{0,4}\Sigma)^{*})$ is called \textit{Bryant's quartic differential}. Surfaces for which $V$ admits a second envelope and $\mathcal{Q}$ is divergence free are called \textit{Voss surfaces}\footnote{These are not to be confused with the ``surfaces of Voss'' studied by Eisenhart~\cite{E1914i}.}. 
\end{definition}

Suppose that $V$ admits a second envelope $\hat{f}$ and define $U := f\oplus\hat{f}$. Then in terms of the splitting of Section~\ref{sec:confgeom}, \eqref{eqn:VWeyl} implies that
\[ \mathcal{D}^{V} = D - \beta - Q- \hat{\beta}_{0} \quad \text{and}\quad \mathcal{N}^{V} = \textrm{I\!I}.\]
Suppose that $Z\in \Gamma T^{1,0}\Sigma$. Then $\mathcal{N}^{V}_{Z} =\textrm{I\!I}^{1,0}_{Z}\in \Gamma (U_{-}\wedge V^{\perp})$ and 
\[ \mathcal{D}^{V}_{Z}\mathcal{N}^{V}_{Z} = D_{Z}\textrm{I\!I}_{Z} - [\beta_{Z}, \textrm{I\!I}_{Z}] - [Q_{Z}, \textrm{I\!I}_{Z}]- [(\hat{\beta}_{0})_{Z},\textrm{I\!I}_{Z}].\]
Now $D_{Z}\textrm{I\!I}_{Z}\in \Gamma (U_{-}\wedge V^{\perp})$, $[\beta_{Z}, \textrm{I\!I}_{Z}] \in \Gamma (\hat{f}\wedge V^{\perp})$, $[Q_{Z}, \textrm{I\!I}_{Z}]=0$ and $ [(\hat{\beta}_{0})_{Z},\textrm{I\!I}_{Z}] \in \Gamma (f\wedge V^{\perp})$. Hence, 
\begin{equation}
\label{eqn:QZZZZ}
\mathcal{Q}(Z,Z,Z,Z) = 
2( [\beta_{Z}, \textrm{I\!I}_{Z}],  [(\hat{\beta}_{0})_{Z},\textrm{I\!I}_{Z}]) = 2( [[\beta_{Z}, \textrm{I\!I}_{Z}],\textrm{I\!I}_{Z}],  (\hat{\beta}_{0})_{Z}),
\end{equation}
using that the Killing metric is $ad$-invariant. Assume now that $(z,\bar{z})$ are local null coordinates for $f$ and thus $\frac{\partial}{\partial z} \in \Gamma T^{1,0}\Sigma$ and $\frac{\partial}{\partial \bar{z}} \in \Gamma T^{0,1}\Sigma$. 
Let 
\[ \mathcal{Q}^{4,0} := \mathcal{Q}(\tfrac{\partial}{\partial z},\tfrac{\partial}{\partial z},\tfrac{\partial}{\partial z},\tfrac{\partial}{\partial z}) \quad\text{and}\quad \mathcal{Q}^{0,4} := \mathcal{Q}(\tfrac{\partial}{\partial \bar{z}} ,\tfrac{\partial}{\partial \bar{z}} ,\tfrac{\partial}{\partial \bar{z}} ,\tfrac{\partial}{\partial \bar{z}}) .\]

\begin{lemma}
\label{lem:Qvanish}
Fix $p\in \Sigma$. Then $\mathcal{Q}^{4,0}(p)= 0$ if and only if $\hat{\beta}_{0}^{1,0}(p)=0$ or $\textrm{I\!I}_{\frac{\partial}{\partial z}}V(p)$ is an isotropic subspace of $V^{\perp}(p)\otimes\mathbb{C}$. Similarly, $\mathcal{Q}^{0,4}(p) = 0$ if and only if $\hat{\beta}_{0}^{0,1}(p)=0$ or $\textrm{I\!I}_{\frac{\partial}{\partial \bar{z}}}V(p)$ is an isotropic subspace of $V^{\perp}(p)\otimes\mathbb{C}$.
\end{lemma}
\begin{proof}
A straightforward calculation shows that $[[\beta_{\frac{\partial}{\partial z}}, \textrm{I\!I}_{\frac{\partial}{\partial z}}],\textrm{I\!I}_{\frac{\partial}{\partial z}}] \in \Gamma (\hat{f}\wedge U_{-})$ and vanishes if and only if $\textrm{I\!I}_{\frac{\partial}{\partial z}}V$ is an isotropic subbundle of $V^{\perp}\otimes\mathbb{C}$. Since $\hat{\beta}^{1,0}_{0}\in \Omega^{1,0}(f\wedge U_{+})$, the first part of the result follows from~\eqref{eqn:QZZZZ}. An analogous analysis can be applied to $\mathcal{Q}^{0,4}$ to complete the proof. 
\end{proof}

Recalling Definition~\ref{def:superconformal} we arrive at the following proposition:

\begin{proposition}
Suppose that $f$ is nowhere half-superconformal, nowhere superconformal and that the central sphere congruence of $f$ admits a second envelope. Then $\mathcal{Q} \equiv 0$ if and only if $\hat{\beta}_{0}=0$. 
\end{proposition}

Proposition~\ref{prop:minimal}, Proposition~\ref{prop:excquasi} and Proposition~\ref{prop:CMCholo} then yield the following corollary: 

\begin{corollary}
Suppose that $f$ is nowhere half-superconformal and that the central sphere congruence of $f$ admits a second envelope. Then $\mathcal{Q}\equiv 0$ if and only if, away from a nowhere dense subset of $\Sigma$, $f$ is either 
\begin{itemize}
\item superconformal
\item a minimal surface in $\mathbb{R}^{p,q}$, 
\item an exceptional quasi-umbilical surface, or 
\item a CMC-$1$ surface in a 3-dimensional space form with sectional curvature $\kappa = -1$. 
\end{itemize}
\end{corollary}

$\mathcal{Q}$ is divergence free if and only if $\mathcal{Q}^{4,0}_{\bar{z}} =\mathcal{Q}^{0,4}_{z}=0$. From~\eqref{eqn:Dbeta} and~\eqref{eqn:DII} we have that 
\[ D_{\frac{\partial}{\partial \bar{z}}}\beta_{\frac{\partial}{\partial z}} = D_{\frac{\partial}{\partial \bar{z}}}
\textrm{I\!I}_{\frac{\partial}{\partial z}}=0.\]
Hence, 
\begin{equation}
\label{eqn:Q40barz}
 \mathcal{Q}^{4,0}_{\bar{z}} = 2( [[\beta_{\frac{\partial}{\partial z}}, \textrm{I\!I}_{\frac{\partial}{\partial z}}],\textrm{I\!I}_{\frac{\partial}{\partial z}}],D_{\frac{\partial}{\partial \bar{z}}}((\hat{\beta}_{0})_{\frac{\partial}{\partial z}})).
\end{equation}
Similarly, we have that 
 \[\mathcal{Q}^{0,4}_{z} = 2( [[\beta_{\frac{\partial}{\partial \bar{z}}}, \textrm{I\!I}_{\frac{\partial}{\partial \bar{z}}}],\textrm{I\!I}_{\frac{\partial}{\partial \bar{z}}}],D_{\frac{\partial}{\partial z}}((\hat{\beta}_{0})_{\frac{\partial}{\partial \bar{z}}})).\]
Applying similar arguments as those in the proof of Lemma~\ref{lem:Qvanish}, we have that:

\begin{lemma}
\label{lem:Qholo}
Fix $p\in \Sigma$. $\mathcal{Q}^{4,0}_{\bar{z}}(p) = 0$ if and only if $D_{\frac{\partial}{\partial \bar{z}}}((\hat{\beta}_{0})_{\frac{\partial}{\partial z}})(p) = 0$ or $\textrm{I\!I}_{\frac{\partial}{\partial z}}V(p)$ is an isotropic subspace of $V^{\perp}(p)\otimes\mathbb{C}$. Similarly, $\mathcal{Q}^{0,4}_{z} (p)= 0$ if and only if $D_{\frac{\partial}{\partial z}}((\hat{\beta}_{0})_{\frac{\partial}{\partial \bar{z}}}) (p)= 0$ or $\textrm{I\!I}_{\frac{\partial}{\partial \bar{z}}}V(p)$ is an isotropic subspace of $V^{\perp}(p)\otimes\mathbb{C}$.
\end{lemma}

\begin{theorem}
Suppose that $f$ is nowhere half-superconformal and nowhere superconformal. Then $f$ is a Voss surface if and only if $f$ is an orthogonal surface of a harmonic congruence of codimension 2 spheres. 
\end{theorem}
\begin{proof}
Suppose that $\hat{f}$ is a second envelope of the central sphere congruence $V$ of $f$ and define $U:=f\oplus \hat{f}$. By Lemma~\ref{lem:Qholo}, $\mathcal{Q}$ is divergence free if and only if 
\[D_{\frac{\partial}{\partial \bar{z}}}((\hat{\beta}_{0})_{\frac{\partial}{\partial z}}) = D_{\frac{\partial}{\partial z}}((\hat{\beta}_{0})_{\frac{\partial}{\partial \bar{z}}}) = 0,\]
i.e., $d^{D}\hat{\beta}_{0} = 0$. Since $A \equiv 0$, we learn from~\eqref{eqn:DQDhatbeta0AII} that $d^{D}Q = 0$ if and only if $d^{D}\hat{\beta}_{0} = 0$. The result then follows from Lemma~\ref{lem:harmcond}. 
\end{proof}

We may now apply Theorem~\ref{thm:orthharmonic} to obtain a characterisation of Voss surfaces: 

\begin{corollary}
Nowhere half-superconformal Voss surfaces in the $(p,q)$-sphere are characterised, away from a nowhere dense subset of their domain, as either 
\begin{itemize}
\item superconformal surfaces 
\item $S$-Willmore surfaces,
\item quasi-umbilical surfaces, or
\item constant mean curvature surfaces in 3-dimensional space forms. 
\end{itemize}
\end{corollary}

\bibliographystyle{plain}
\bibliography{bibliography2017}

\end{document}